\documentclass{amsart}
\usepackage{amsmath,amssymb,enumerate}
\allowdisplaybreaks[2]

\newcommand{\IB}{{\mathbb B}}
\newcommand{\IC}{{\mathbb C}}
\newcommand{\IM}{{\mathbb M}}

\newcommand{\IZ}{{\mathbb Z}}
\newcommand{\IR}{{\mathbb R}}

\newcommand{\cC}{{\mathcal C}}

\newcommand{\cZ}{{\mathcal Z}}
\newcommand{\acts}{\curvearrowright}
\newcommand{\ve}{\varepsilon}

\newcommand{\Ga}{\Gamma}
\newcommand{\La}{\Lambda}
\newcommand{\id}{\mathrm{id}}
\newcommand{\SL}{\mathrm{SL}}

\DeclareMathOperator{\Prob}{Prob}

\newcommand{\ip}[1]{\mathopen{\langle}#1\mathclose{\rangle}}
\newtheorem{thm}{Theorem}

\newtheorem{cor}[thm]{Corollary}
\newtheorem{lem}[thm]{Lemma}
\newtheorem*{claim}{Claim}
\theoremstyle{definition}
\newtheorem{defn}[thm]{Definition}
\newtheorem{exa}[thm]{Example}
\newtheorem{rem}[thm]{Remark}
\title[Fullness of some von Neumann algebras]{A remark on fullness of some group measure space von Neumann algebras}
\author{Narutaka Ozawa}
\email{narutaka@kurims.kyoto-u.ac.jp}
\address{RIMS, Kyoto University, \mbox{606-8502} Kyoto, Japan}
\thanks{The author was partially supported by JSPS26400114}
\date{\today}

\begin{document}
\begin{abstract}
Recently C. Houdayer and Y. Isono have proved among other things 
that every biexact group $\Gamma$ has the property that for any 
non-singular strongly ergodic essentially free action 
$\Gamma \curvearrowright (X,\mu)$ 
on a standard measure space, the group measure space von Neumann 
algebra $\Gamma \ltimes L^\infty(X)$ is full.
In this note, we prove the same property for a wider class of groups, 
notably including $\mathrm{SL}(3,{\mathbb Z})$. 
We also prove that for any connected simple Lie group $G$ with finite center, 
any lattice $\Gamma \le G$, and any closed non-amenable subgroup $H \le G$, 
the non-singular action $\Gamma \curvearrowright G/H$ is strongly ergodic and 
the von Neumann factor $\Gamma \ltimes L^\infty(G/H)$ is full. 
\end{abstract}
\maketitle
\section{Introduction}
Recall that a von Neumann factor $N$ is said to be \emph{full} (\cite{connes}) 
if the \emph{central sequence algebra} $N'\cap N^\omega$ is trivial for a 
non-principal ultrafilter $\omega$.
For more information on fullness, we refer the reader to \cite{ah,hi} and the references therein. 
Recently, C. Houdayer and Y. Isono (\cite{hi}) have studied which group has 
the property that the group measure space factor $\Gamma\ltimes L^\infty(X)$ 
is full for every non-singular strongly ergodic essentially free action $\Ga\acts(X,\mu)$ 
on a standard measure space, and they have proved among other things 
that \emph{biexact} groups (e.g., hyperbolic groups) have this property. 
Recall that a non-singular action $\Ga\acts(X,\mu)$ on a probability space 
(in case $\mu$ is not a probability measure, replace it with a probability 
measure in the same measure class) is said to be \emph{strongly ergodic} 
if any sequence $(E_n)_n$ of measurable subsets such that 
$\mu(E_n\bigtriangleup sE_n)\to0$ for every $s\in\Ga$ 
has to be trivial: $\mu(E_n)(1-\mu(E_n))\to0$. 
Unless the action is strongly ergodic, the von Neumann algebra
$\Gamma\ltimes L^\infty(X,)$ cannot be full. 
We note that in case the strongly ergodic action $\Ga\acts(X,\mu)$ 
is \emph{probability measure preserving}, M. Choda (\cite{choda}) has 
already obtained in 1982 a rather satisfactory result 
that the factor $\Gamma\ltimes L^\infty(X)$ 
is full whenever $\Ga$ is non \emph{inner amenable}. 
In this note, we combine Choda's proof with R. J. Zimmer's notion of amenable 
action (\cite{zimmer:invent}) 
and prove Houdayer and Isono's property for a wider class of groups, notably 
including $\SL(3,\IZ)$, which is not biexact (\cite{sako}).
We also prove that for any connected simple Lie group $G$ with finite center, 
any lattice $\Ga\le G$, and any closed non-amenable subgroup $H\le G$, 
the non-singular action $\Ga\acts G/H$ is strongly 
ergodic and the von Neumann algebra $\Ga\ltimes L^\infty(G/H)$ is full 
as long as it is a factor. 

\subsection*{Acknowledgements}
The author thanks Professor R.~Tomatsu for bringing his attention 
to the motivating problem about fullness of $\SL(n,\IZ)\ltimes L^\infty(\IR^n)$ 
and for providing him the first part of the proof of Corollary~\ref{cor}. 
He also thanks Professor P.~de la Harpe and Professor N.~Monod 
for informing him several years ago about Example~\ref{exa:sl3}, 
and thanks Professor A.~Ioana and Professor J.~Peterson for 
communicating to the author Remark~\ref{rem:ip}.
\section{Groups with Houdayer and Isono's property}
We recall the notion of amenability in the sense of Zimmer and generalize it to a relative situation. 
The definition below is different from the original, but is equivalent to 
it (\cite[Proposition 4.1]{zimmer:invent}). 
We also note that amenability of a non-singular action 
$\Ga\acts(X,\mu)$ is equivalent to injectivity of the von Neumann algebra $\Ga\ltimes L^\infty(X)$ 
(\cite{zimmer:invent,delaroche}).

\begin{defn}\label{defn}
Let $\Ga$ be a discrete group and $\cC$ be a non-empty family of its subgroups, and 
consider the set $K=\bigsqcup_{\Lambda\in\cC} \Ga/\Lambda$ on which $\Ga$ 
acts by translation. 
A non-singular action $\Ga\acts(X,\mu)$ of $\Ga$ 
on a standard measure space $(X,\mu)$ 
is said to be \emph{amenable} (in the sense of Zimmer) 
if there is a $\Ga$-equivariant conditional expectation $\Phi$ from 
$L^\infty(\Ga\times X)$ onto $L^\infty(X)$, where $\Ga$ acts on $\Ga\times X$ diagonally. 
(Recall that $\Phi$ is called a 
\emph{conditional expectation} if it is positive and satisfies $\Phi(1\otimes f)=f$ for every $f\in L^\infty(X)$.)
We say $\Ga\acts(X,\mu)$ is \emph{amenable relative to $\cC$} 
if there is a $\Ga$-equivariant conditional expectation from 
$L^\infty(K\times X)$ onto $L^\infty(X)$, where $\Ga$ acts on $K\times X$ diagonally. 
\end{defn}

\begin{rem}
We note that amenable actions on non-atomic measure spaces are never strongly ergodic 
by the Connes--Feldman--Weiss theorem (\cite{cfw}, see also \cite[Theorem 2.4]{schmidt}).
We collect here some simple observations. 
Amenability does not change when one replaces 
the measure $\mu$ with another in the same measure class. 
By definition, amenability is same as amenability relative to $\{\mathbf{1}\}$, where $\mathbf{1}$ 
is the trivial subgroup consisting of the neutral element $e$. 
Let $\Ga\acts(X,\mu)$ and $\cC$ be given. 
If $\Ga\acts(X,\mu)$ is amenable, then it is amenable relative to $\cC$. 
The converse also holds true if the family $\cC$ consists entirely of amenable subgroups. 
If $\ell_\infty(\bigsqcup_{\Lambda\in\cC} \Ga/\Lambda)$ admits a $\Ga$-invariant state, 
then the action $\Ga\acts(X,\mu)$ is amenable relative to $\cC$. 
The converse also holds true if $L^\infty(X)$ admits a $\Ga$-invariant state 
(e.g., if the action is probability measure preserving). 
Let $K$ be a set on which $\Ga$ acts. 
Then, there is a $\Ga$-equivariant conditional 
expectation from $L^\infty(K\times X)$ onto $L^\infty(X)$ 
if and only if 
$\Ga\acts(X,\mu)$ is amenable relative to the family $\{ \Ga_p : p\in K\}$ 
of the stabilizer subgroups $\Ga_p=\{ s\in\Ga : sp=p\}$. 
We will be interested in the conjugation action of $\Ga$ on $\Ga$. 
The stabilizer subgroup of the conjugation action $\Ga\acts\Ga$ at $t\in\Ga$ 
is the centralizer subgroup $C_{\Ga}(t):=\{ s\in\Ga : st=ts \}$. 
\end{rem}

\begin{thm}\label{thm}
Let $\Ga$ be a countable group, 
$\Omega\subset\Ga$ be a conjugacy invariant subset, 
and $\cC_\Omega=\{ C_\Ga(t) : t\in\Omega\}$. 
Let $\Ga\acts(X,\mu)$ be a non-singular action of 
$\Ga$ on a standard measure space, which does not have 
a non-null $\cC_\Omega$-relatively amenable component. 
Let $(w_n)_n$ be a bounded sequence in $\Ga\ltimes L^\infty(X)$ such that 
$w_n-\lambda(s)w_n\lambda(s)^*\to0$ ultrastrongly for every $s\in\Ga$, and 
expand them as $w_n=\sum_t \lambda(t)w_n^t$.
Then, one has $\sum_{t\in\Omega} |w_n^t|^2\to0$ ultrastrongly in $L^\infty(X)$.
\end{thm}

\begin{proof}
We put $s\cdot f:=\lambda(s)f\lambda(s)^*$ for $f\in L^\infty(X)$ and $s\in\Ga$, 
or equivalently $(s\cdot f)(x)=f(s^{-1}x)$. 
We may assume that $\mu$ is a probability measure. 
Put $h_n^t := |w_n^t|^2$ and $C:=\sup_n\| w_n\|^2$. 
Then $h_n^t$'s are non-negative functions such that 
$0\le \sum_t h_n^t \le C$ and 
\begin{align*}
\sum_{t\in\Ga} & \| h_n^t - s\cdot h_n^{s^{-1}ts} \|_{L^1(X)} \\
 &\le (\sum_{t\in\Ga} \| w_n^t - s\cdot w_n^{s^{-1}ts} \|_{L^2(X)} ^2)^{1/2} 
        (\sum_{t\in\Ga} \| |w_n^t| + |s\cdot w_n^{s^{-1}ts}| \|_{L^2(X)} ^2)^{1/2} \\
 &\le \| (w_n-\lambda(s)w_n\lambda(s)^*)(\delta_e\otimes 1_X)\|_{\ell_2(\Ga\times X)}
        \cdot 2\|w_n\| \\
 &\to 0
\end{align*}
for every $s\in\Ga$. It suffices to prove that if there are functions $h_n^t$ as above 
such that $h_n:=\sum_{t\in\Omega} h_n^t$ does not converge to zero, then 
a non-null $\cC_\Omega$-relatively amenable component exists. 
By compactness, we may assume that ultraweak limit 
$h:=\lim_n h_n$ exists and after scaling that $X_0:=\{ x \in X : 1\le h(x)\le 2\}$ 
has positive measure. 
By passing to convex combinations, we may further assume 
that $h_n\to h$ ultrastrongly. 
Since $\Omega$ is conjugacy invariant, the function $h$ is $\Ga$-invariant and so is $X_0$.
We consider the conditional expectations 
$\Phi_n$ from $L^\infty(\Omega\times X_0)$ onto $L^\infty(X_0)$ given by 
\[
\Phi_n(f)(x)=\frac{1}{h_n(x)}\sum_{t\in\Omega}h_n^t(x)f(t,x),
\]
and claim that it is approximately $\Ga$-equivariant. 
(Let us not care the points $x$ where $h_n(x)=0$.)
Let $s\in\Ga$ and $f\in L^\infty(\Omega\times X_0)$ with $\|f\|_\infty\le1$ be given. 
Then one has 
\[
\ve_n:=\mu(\{ x\in X_0 : |h_n(x)-h(x)| + |h_n(s^{-1}x) - h(x)| \geq 1/n \})\to 0
\]
(recall that $h$ is $\Ga$-invariant) and  
\begin{align*}
&\| \Phi_n(s\cdot f) - s\cdot\Phi_n(f)\|_{L^1(X_0)}\\
 &\le \int_{X_0}\frac{1}{h(x)}\Bigl|\sum_{t\in\Omega} h_n^t(x) f(s^{-1}ts,s^{-1}x) 
  - \sum_{t\in\Omega} h_n^t(s^{-1}x) f(t,s^{-1}x) \Bigr|\,d\mu(x) + 2\ve_n+\frac{2}{n}\\
 &\le \int_{X_0}\sum_{t\in\Omega}\bigl| h_n^{t}(x) - (s\cdot h_n^{s^{-1}ts})(x) \bigr| 
  |f(s^{-1}ts,s^{-1}x)|\,d\mu(x)  + 2\ve_n+\frac{2}{n}\\
 &\to 0.
\end{align*}
This proves the claim. 
Hence any pointwise ultraweak limit of $(\Phi_n)_n$ 
is a $\Ga$-equivariant conditional expectation and the action 
$\Ga\acts(X_0,\mu|_{X_0})$ is amenable relative to $\cC_\Omega$.
\end{proof}

The following Corollary gives a criterion of $\Ga$ for which every 
non-singular strongly ergodic essentially free action $\Ga\acts(X,\mu)$ 
gives rise to a full factor $\Ga\ltimes L^\infty(X)$. The essential freeness 
assumption is only used to assure $\Ga\ltimes L^\infty(X)$ (and perhaps 
$\La\ltimes L^\infty(X_0)$ in the proof) is a factor and probably can be 
greatly relaxed. 

\begin{cor}\label{cor}
Assume that $\Ga$ is a countable group which has a finite index 
subgroup $\La$ such that $\La_{\mathrm{nac}}:= \{ t\in\La : C_\La(t)\mbox{ is not amenable}\}$
is finite. 
Then, for any non-singular strongly ergodic essentially free action  
$\Ga\acts(X,\mu)$ on a standard measure space, 
the von Neumann factor $\Ga\ltimes L^\infty(X)$ is full.
\end{cor}

\begin{proof}
Let $N=\Ga\ltimes L^\infty(X)$ and $M=\La\ltimes L^\infty(X)$. 
Then, $M\subset N$ is a finite index inclusion with 
a canonical normal conditional expectation $E$ from $N$ onto $M$. 
We first observe that it suffices to show that $M'\cap M^\omega$ 
is finite-dimensional for a non-principal ultrafilter $\omega$. 
(We refer the reader to \cite[Section 2]{hi} for an account of 
the central sequence algebra $M'\cap M^\omega$.)
Indeed, if it is so, then $N'\cap N^\omega$ is finite dimensional 
since the map $E^\omega$ from 
$N'\cap N^\omega$ into $M'\cap M^\omega$ satisfies 
the Pimsner--Popa inequality (\cite{pp}). 
Since $N$ is a factor, this implies that $N$ is full by \cite[Corollary 2.6]{hr}. 

Thus we are left to show that $M$ is a direct sum of finitely many full factors 
(and hence $M'\cap M^\omega$ is finite-dimensional). 
For this, we note that if $\Lambda\acts(X_0,\mu_0)$ is a strongly ergodic 
and essentially free action, then the crossed product 
$M_0:=\La\ltimes L^\infty(X_0,\mu_0)$ is a full factor. 
Indeed, by \cite[Lemma 5.1]{hi}, if $M_0$ were not full, 
then there would be a unitary central sequence $(w_n)_n$ in $M_0$ 
such that $\sum_{t\in F}|w_n^t|^2\to0$ ultrastrongly for every finite 
subset $F\subset\La$. 
But then Theorem~\ref{thm}, applied to $\Omega=\La\setminus\La_{\mathrm{nac}}$ 
(note that we may assume that $(X_0,\mu_0)$ is non-atomic and 
thus $\La$-action on it is non-amenable), 
implies that 
$1_{M_0}=\sum_{t\in \La}|w_n^t|^2\to0$, which is absurd. 
This proves $M_0$ is a full factor. Therefore the proof of Corollary 
is complete once we prove the following claim, 
which is probably known to experts. 
\begin{claim}
The restriction $\La\acts(X,\mu)$ of the strongly ergodic action of 
$\Ga$ to a finite index subgroup $\La$ decomposes into finitely many ergodic components 
and each of the ergodic components is strongly ergodic.
\end{claim}

If $\La\acts(X,\mu)$ is not a union of finitely many ergodic components, 
then for any $\ve>0$, there is a $\La$-invariant measurable subset $E\subset X$ 
such that $0<\mu(E)<\ve$. 
Also, if there is a non-null ergodic component $X_0\subset X$ 
which is not strongly ergodic, then for any $\ve>0$, 
there is an asymptotically $\Lambda$-invariant sequence $(E_n)_n$ 
of measurable subsets of $X_0$ 
such that $0< \lim_n\mu(E_n) <\ve$ (see, e.g., the proof of \cite[Lemma 2.3]{js}). 
Therefore, to prove Claim, it suffices to show that there is $\ve>0$ such that 
any asymptotically $\Lambda$-invariant sequence $(E_n)_n$ 
with $\inf_n\mu(E_n)>0$ satisfies $\liminf_n\mu(E_n)\geq\ve$.
Now, let $\{ t_0,\ldots,t_d\}$ be a system of representatives 
of the left cosets $\Ga/\La$, and take $\ve>0$ such that 
$\mu(E)\le\ve$ implies $\sum_i\mu(t_iE)\le1/2$. 
Then for any asymptotically $\La$-invariant sequence $(E_n)_n$ 
with $\inf_n\mu(E_n)>0$, 
the functions $f_n:=\sum_i 1_{t_i E_n}$ satisfy 
$\|s\cdot f_n-f_n \|_{L^1(\mu)}\to0$ for every $s\in\Ga$, 
and so by strong ergodicity, $\| f_n - \int f_n\,d\mu\|_{L^1(\mu)}\to0$. 
%
%
This implies $\liminf\mu(E_n)\geq\ve$ and Claim is proved. 
\end{proof}

\begin{exa}\label{exa:hyp}
If $\Ga$ is a torsion-free hyperbolic group or a torsion-free discrete 
subgroup of a simple Lie group of rank one, then $\Ga_{\mathrm{nac}}\subset\{e\}$. 
Indeed, any non-trivial centralizer subgroup is elementary and hence is abelian. 
More generally, if $\Ga$ is a group acting with finite quotient on a fine 
hyperbolic graph $K$ in such a way that 
every vertex stabilizer is amenable and the neutral element is the only element 
which fixes infinitely many points in the boundary $\partial K$ of $K$, 
then $\Ga_{\mathrm{nac}}\subset\{e\}$. 
For the proof recall that the Bowditch compactification $\Delta K := K\cup\partial K$ 
is topologically amenable (\cite[Theorem 1]{relhyp}). 
Hence any subgroup $\La\le\Ga$ which admits a non-empty $\La$-invariant 
finite subset of $\Delta K$ is amenable. 
Let $t\in\Ga$ be such that $C_\Ga(t)$ is non-amenable and 
denote by $\mathrm{Fix}(t)$ the $t$-fixed points of $\Delta K$. 
Since $C_\Ga(t)$ has unbounded orbits in $K$, 
the closed $C_\Ga(t)$-invariant subset $\mathrm{Fix}(t)$ contains 
infinitely many boundary points.  By assumption $t=e$. 
This applies to a torsion-free relatively hyperbolic group and an amalgamated 
free product $\Ga:=\Ga_1*_\Lambda\Ga_2$ of amenable groups $\Ga_i$ 
over a common malnormal subgroup $\Lambda$.
\end{exa}

\begin{exa}
If $\Ga$ is a subgroup of $\SL(2,F)$ for a field $F$, then 
$\Ga_{\mathrm{nac}}\subset\{\pm I\}$. 
Indeed, by considering the Jordan normal form, it is not hard to see that 
the centralizer subgroup $C_{\SL(2,F)}(g)$ of 
any element $g\in\SL(2,F)\setminus\{\pm I\}$ is abelian.
\end{exa}

\begin{exa}\label{exa:sl3}
If $\La$ is a subgroup of $\SL(3,\IZ)$ which contains no elements of order $2$, 
then $\La_{\mathrm{nac}}\subset\{e\}$. In particular, the finite-index subgroup 
$\Ga(3)=\ker(\SL(3,\IZ)\to\SL(3,\IZ/3\IZ))$ satisfies 
$\Ga(3)_{\mathrm{nac}}=\{e\}$. 
For a proof, suppose that $g\in\SL(3,\IZ)$ has 
a non-abelian centralizer subgroup in $\SL(3,\IR)$. 
Then, $g$ must be diagonalizable in $\SL(3,\IR)$ 
and has eigenvalues $(\lambda,\lambda,\lambda^{-2})$ for some $\lambda\in\IC$. 
Its characteristic polynomial $p(t)=t^3+at^2+bt-1$ has 
integer coefficients and satisfies 
$0=9p(\lambda) - (3\lambda+a)p'(\lambda) = 2(3b-a^2)\lambda - (9 + ab).$
So, if $3b-a^2\neq 0$, then $\lambda$ is rational.  
If $3b-a^2=0$, then $0=3p'(\lambda)=(3\lambda + a)^2$ and $\lambda$ is again rational. 
Hence in either case $p(\lambda)=0$ implies $\lambda\in\{\pm1\}$.
This proves $g^2=I$. 
Now, suppose for a contradiction that there is an element $g\in\Ga(3)$ of order $2$. 
Then, the largest integer $m$ such that $g-I \in 3^m\IM_3(\IZ)$ is positive, 
but this is in contradiction with $(g-I)^2=-2(g-I)$.
In fact it is well-known that the finite index subgroups 
$\Gamma(m)$ are torsion-free for all $m\geq3$. 
By Corollary~\ref{cor}, the factor $\SL(3,\IZ)\ltimes L^\infty(X)$ is full for every 
strongly ergodic essentially free action $\SL(3,\IZ)\acts(X,\mu)$. 
It is not clear if the same conclusion holds for $\SL(n\geq4,\IZ)$.
\end{exa}

\begin{exa}
In Corollary~\ref{cor} one can replace the condition that 
$\Lambda_{\mathrm{nac}}$ is finite with existence of a map $\zeta\colon \La\to \Prob(\La)$ 
such that $\lim_{t\to\infty}\|\zeta_{sts^{-1}} - s\zeta_t\|=0$ for every $s\in\La$. 
(In fact the latter condition is more general.) 
Indeed, if there is a unitary central sequence $(w_n)_n$ in $\La\ltimes L^\infty(X)$ 
such that $\sum_{t\in F}|w_n^t|^2\to0$ for every finite subset $F\subset\La$, 
then the action is amenable since the maps 
$\Phi_n\colon L^\infty(\La\times X) \to L^\infty(X)$ given by 
$\Phi_n(f)(x)=\sum_{t,p}|w_n^t(x)|^2\zeta_t(p)f(p,x)$ are approximately $\La$-equivariant.
We note that biexact groups satisfy the above property by Proposition 4.1 in \cite{kurosh}. 
\end{exa}
\section{Actions of lattices on homogeneous spaces}
In this section we consider the non-singular action 
$\Ga\acts G/H$ of a lattice $\Ga$ in a second countable locally compact group $G$ 
on a homogeneous space $G/H$. Here $H\le G$ is a closed subgroup and 
$G/H$ is equipped with a $G$-quasi-invariant measure, which is unique up to equivalence.
Generally speaking, one can relate the action $\Ga\acts G/H$ to 
the action $\Ga\backslash G\curvearrowleft H$ (see \cite[Section 3]{pv}). 
For example, if one is amenable, then so is the other (\cite[Remark 4.2]{pv}).
The following is the relative version of this fact. 

\begin{lem}\label{lem:ra}
Let $\Ga\acts G/H$ be as above and assume that it is amenable relative 
to a non-empty family $\cC$ of subgroups of $\Ga$. 
Then, there is an $H$-invariant state $\psi$ on 
$L^\infty(\bigsqcup_{\Lambda\in\cC} \La\backslash G)$, 
where $H$ acts on $\bigsqcup_{\Lambda\in\cC} \La\backslash G$ diagonally from the right. 
Moreover, if there is a normal $\Ga$-equivariant conditional expectation from 
$L^\infty(K\times G/H)$ onto $L^\infty(G/H)$, where $K=\bigsqcup_{\La\in\cC} \Ga/\La$, 
then the $H$-invariant state $\psi$ can be taken normal. 
\end{lem}
\begin{proof}
Let $\Phi$ denote a $\Ga$-equivariant conditional expectation 
from $L^\infty(K\times G/H)$ onto $L^\infty(G/H)$. 
We fix a lifting $\sigma\colon K\to \Ga$ such that $\sigma(p)\Lambda = p$ for 
$p\in\Ga/\La\subset K$ and a measurable lifting $\tau\colon G/H\to G$. 
The corresponding measure space isomorphisms and cocycles are denoted as follows: 
\begin{gather*}
\theta\colon K\times G/H \times H \to K \times G,\quad 
\theta(p,x,y) = (p,\sigma(p)^{-1}\tau(x)y);\\
\theta'\colon G\to G/H \times H,\quad \theta'(g)=(gH,\tau(gH)^{-1}g);\\
\alpha\colon\Ga\times K\ni(s,p)\mapsto \sigma(sp)^{-1}s\sigma(p) \in \Lambda 
\mbox{ for $s\in\Ga$ and $p\in\Ga/\La\subset K$};\\
\beta\colon \Ga\times G/H \ni(s,x)\mapsto \tau(sx)^{-1}s\tau(x)\in H.
\end{gather*}
These maps satisfy the following relations. 
Let $s\cdot (p,x,y) \cdot h := (sp,sx,\beta(s,x)yh)$ for $(p,x,y)\in K\times G/H \times H$ 
and $s\cdot (p,g)\cdot h := (sp,\alpha(s,p)gh)$ for $(p,g)\in K\times G$.
Then, one has $\theta( s\cdot (p,x,y) \cdot h ) = s\cdot \theta(p,x,y)\cdot h$. 
Let $s\cdot(x,y)\cdot h := (sx,\beta(s,x)yh)$ for $(x,y)\in G/H\times H$. 
Then, one has $\theta'(sgh) = s\cdot\theta'(g)\cdot h$.

We claim that the conditional expectation 
$\Phi\otimes\id_{L^\infty(H)}$ from $L^\infty(K\times G/H\times H)$ 
onto $L^\infty(G/H\times H)$ satisfies 
$(\Phi\otimes\id_{L^\infty(H)})(s\cdot f\cdot h)
 = s\cdot (\Phi\otimes\id_{L^\infty(H)})(f)\cdot h$ for 
every $s\in\Ga$ and $h\in H$, 
where $(s\cdot f\cdot h)(p,x,y):=f(s^{-1}\cdot (p,x,y) \cdot h^{-1})$, etc.
Note that even though $\Phi$ may not be normal, 
the map $\Phi\otimes\id_{L^\infty(H)}$ is well-defined via the defining relation 
\[
(\xi\otimes\eta)((\Phi\otimes\id_{L^\infty(H)})(f))=\xi(\Phi((\id\otimes\eta)(f)))
\mbox{ for $(\xi,\eta)\in L^\infty(G/H)_*\times L^\infty(H)_*$}.
\]
(Of course $L^\infty(X,\mu)_*=L^1(X,\mu)$ 
under the duality coupling $\ip{f,\xi}=\int f\xi\,d\mu$, 
but we often regard $\xi\in L^1(X,\mu)$ as an ultraweakly 
continuous linear functional on $L^\infty(X)$.) 
We observe that for any countable partition $(q_i)_i$ of unity in $L^\infty(G/H)$ 
and any bounded sequence $(f_i)_i$ in $L^\infty(K\times G/H)$, one has 
$\Phi(\sum_i f_i (1_K\otimes q_i)) = \sum_i\Phi(f_i)q_i$.
Indeed, this follows from the fact that for any state $\xi\in L^\infty(G/H)_*$ one has 
\begin{align*}
|\xi(\Phi(\sum_{i\geq n} f_i (1_K\otimes q_i)))| 
 &\le (\xi\circ\Phi)( \bigl| \sum_{i\geq n} f_i (1_K\otimes q_i) \bigr| )\\
 &\le (\xi\circ\Phi)( (\sup_i\|f_i\|) \sum_{i\geq n}1_K\otimes q_i)\\
  &= (\sup_i\|f_i\|)\xi(\sum_{i\geq n}q_i)\to0.
\end{align*}
Now, let $s\in\Ga$ be given and take $\eta\in L^1(H,\nu)$ and $\ve>0$ arbitrary. 
Here $\nu$ denotes the left Haar measure of $H$. 
Let $(Y_i)_i$ be a countable measurable partition of $H$ together 
with $h_i\in Y_i$ such that $\|h\eta - h_i\eta \|_1<\ve$ for $h\in Y_i$, 
and put $X_i=\{ x\in G/H : \beta(s^{-1},x) \in H_i\}$.  
Then, for every $f\in L^\infty(K\times G/H\times H)$ with $\|f\|\le1$, 
one has
\[
\|(\id_{L^\infty(K\times G/H)}\otimes \eta)\bigl(s\cdot f - \sum_i ((s\otimes h_i^{-1})f) 1_{K\times X_i\times H}\bigr)\|_{L^\infty(K\times G/H)}<\ve,
\]
where $s$ acts on $L^\infty(K\times G/H)$ diagonally. 
Indeed, for every $i$ and a.e.\ $(p,x)\in K\times X_i$ one has 
\begin{align*}
((\id\otimes \eta)(s\cdot f))(p,x) 
 &= \int_H f(s^{-1}p,s^{-1}x,\beta(s^{-1},x)y)\eta(y)\,d\nu(y)\\
 &\approx_\ve \int_H f(s^{-1}p,s^{-1}x, h_i y)\eta(y)\,d\nu(y)\\
 &=(\id\otimes \eta)((s\otimes h_i^{-1})f)(p,x).
\end{align*}
Similarly, for every $g\in L^\infty(G/H\times H)$ with $\|g\|\le1$ one has 
\[
\|(\id_{L^\infty(G/H)}\otimes \eta)\bigl(s\cdot g - \sum_i ((s\otimes h_i^{-1})g) 1_{X_i\times H}\bigr)\|_{L^\infty(G/H)}<\ve.
\]
Hence 
\begin{align*}
(\id\otimes \eta)((\Phi\otimes\id)(s\cdot f))
 &= \Phi( (\id\otimes \eta)(s\cdot f) )\\
 &\approx_\ve \Phi( (\id\otimes \eta) (\sum_i ((s\otimes h_i^{-1})f) 1_{K\times X_i\times H}))\\
 &= \sum_i \Phi((\id\otimes \eta) ((s\otimes h_i^{-1})f))1_{X_i}\\
 &= \sum_i (\id\otimes\eta)((s\otimes h_i^{-1})((\Phi\otimes\id)(f)))1_{X_i}\\
 &\approx_\ve (\id\otimes\eta)(s\cdot (\Phi\otimes\id)(f)).
\end{align*}
Since $\eta$ and $\ve>0$ were arbitrary, one obtains 
$(\Phi\otimes\id)(s\cdot f)=s\cdot (\Phi\otimes\id)(f)$. 
That $(\Phi\otimes\id)(f\cdot h)=(\Phi\otimes\id)(f)\cdot h$ 
for $h\in H$ is obvious. 

We define $\Psi\colon L^\infty(K\times G)\to L^\infty(G)$ by 
$\Psi = \theta'_*\circ(\Phi\otimes\id_{L^\infty(H)})\circ\theta_*$. 
It may not be a conditional expectation, 
but it is a unital positive map, which is $(\Ga\times H)$-equivariant: 
\[
\Psi( s\cdot f\cdot h) = s\Psi(f)h 
\mbox{ for $s\in\Ga$, $h\in H$, and $f\in L^\infty(K\times G)$},
\]
where $(s\Psi(f)h)(g)=\Psi(f)(s^{-1}gh^{-1})$ for $g\in G$. 
We denote by $\iota$ the embedding 
of $L^\infty(\bigsqcup_{\La\in\cC} \La\backslash G)$ into 
$L^\infty(K\times G)$ given by 
$\iota(f)(p,g)= f(\Lambda g)$ for $\Lambda\in\cC$, $p\in\Ga/\La\subset K$, and $g\in G$. 
The map $\iota$ is $(\Ga\times H)$-equivariant, where $\Ga$ acts trivially on 
$\bigsqcup_{\La\in\cC} \La\backslash G$ and $H$ acts on it from the right. 
Thus $\Psi\circ\iota$ is an $H$-equivariant unital positive map from 
$L^\infty(\bigsqcup_{\La\in\cC} \La\backslash G)$ into 
$L^\infty(G)^\Ga\cong L^\infty(\Ga\backslash G)$. 
An $H$-invariant state $\psi$ can be obtained 
by composing $\Psi\circ\iota$ with the $G$-invariant probability measure on $\Ga\backslash G$. 
If $\Phi$ is normal, then so is $\psi$.
\end{proof}

\begin{lem}[cf.\ Proposition G in \cite{ioana}]\label{lem:GH}
Let $G$ be a second countable locally compact group, $\Ga\le G$ be a lattice, 
and $H\le G$ be a closed subgroup such that the action 
$\Ga\backslash G\curvearrowleft H$ is strongly ergodic. 
Then $\Ga\acts G/H$ is strongly ergodic. 
\end{lem}
%
%
\begin{proof}
We first recall that a non-singular action $H\acts (Y,\nu)$ of a locally compact group $H$ 
on a probability space $(Y,\nu)$ is said to be strongly ergodic if 
any sequence $(E_n)_n$ of measurable subsets of $Y$ that is approximately $H$-invariant, 
i.e., $\nu(E_n\bigtriangleup hE_n)\to0$ uniformly on compact subsets of $H$, is trivial in the 
sense that $\nu(E_n)(1-\nu(E_n))\to0$. 

Now, take a Borel lifing $\tau\colon \Ga\backslash G\to G$ 
and let $Y=\tau(\Ga\backslash G)$ be the corresponding $\Ga$-fundamental domain. 
Since $G$ is $\sigma$-compact, we may assume that $\tau(L)$ is relatively compact 
for every compact subset $L\subset \Ga\backslash G$. 
The Haar measure $\lambda_G$ of $G$ is normalized so that $\lambda_G(Y)=1$. 
Then, the formula $\lambda(\Ga A)=\lambda_G(\Ga A\cap Y)$ for 
measurable subsets $A\subset G$ defines the $G$-invariant 
probability measure $\lambda$ on $\Ga\backslash G$. 
Assume that there is a non-trivial approximately $\Ga$-invariant 
sequence $(E_n)_n$ of measurable subsets of $G/H$. 
We will prove that $(\Ga(E_nH\cap Y))_n$ is a non-trivial approximately 
$H$-invariant sequence of measurable subsets of $\Ga\backslash G$. 

Recall that $L^\infty(G,\lambda_G)\cong L^\infty(G/H \times H, \mu\otimes\nu)$, 
where $\mu$ and $\nu$ are quasi-invariant probability measures. 
Since $(\mu\otimes\nu)( sE_nH \bigtriangleup E_nH )=\mu(sE_n\bigtriangleup E_n)\to0$ 
for every $s\in\Ga$, one has
$\lambda_G( (sE_nH \bigtriangleup E_nH) \cap Z)\to0$ for 
any $Z$ with $\lambda_G(Z)<\infty$. 
Thus, if $\lambda_G(E_nH\cap Y)\to0$, then 
$\lambda_G(E_nH\cap sY)\approx \lambda_G(sE_nH\cap sY)=\lambda_G(E_nH\cap Y)\to0$
for every $s\in\Ga$, which means that $\mu(E_n)\to0$. 
Therefore $(\Ga(E_nH\cap Y))_n$ is non-trivial if $(E_n)_n$ is non-trivial. 

Let $\beta\colon \Ga\backslash G\times H\ni(y,h)\mapsto \tau(y)h\tau(yh)^{-1}\in\Ga$ be the cocycle associated with $\tau$ that satisfies 
$\tau(y)h=\beta(y,h)\tau(yh)$ for $y\in \Ga\backslash G$ and $h\in H$. 
Let a compact subset $K\subset H$ and $\ve>0$ be given. 
Take a compact subset $L\subset \Ga\backslash G$ 
such that $\lambda(L)>1-\ve$. 
Then, $F:=\{\beta(y,h) : y\in L, h\in K\}$ is a finite subset of $\Ga$, 
since it is relatively compact in $G$. 
Thus, if $n$ is large enough, then one has 
$\lambda_G( (FE_n H\cap Y) \setminus (E_n H\cap Y))<\ve$.
So, for every $h\in K$ one has 
\begin{align*}
\lambda( \Ga(E_n H\cap Y)h \setminus \Ga(E_n H\cap Y) )
 &\approx_\ve \lambda( \Ga(E_n H\cap \tau(L))h \setminus \Ga(E_n H\cap Y) )\\
 &\le \lambda_G( (FE_n H\cap Y) \setminus (E_n H\cap Y))\\
 &\approx_\ve 0.
\end{align*}
This means that $(\Ga(E_nH\cap Y))_n$ is approximately $H$-invariant. 
\end{proof}

\begin{thm}\label{thm2}
Let $G$ be a connected simple Lie group with finite center $\cZ(G)$, 
$\Ga\le G$ be a lattice, and $H\le G$ be a closed non-amenable 
subgroup. 
Then, the action $\Ga\acts G/H$ is strongly ergodic, 
and if $\Ga \cap \cZ(G) \cap H =\{e\}$, then $\Ga\ltimes L^\infty(G/H)$ is a full factor. 
\end{thm}

\begin{proof}
Strong ergodicity of $\Ga\acts G/H$ follows from 
that of $\Ga\backslash G\curvearrowleft H$ by Lemma~\ref{lem:GH}. 
Since the latter action is finite measure preserving, 
strong ergodicity follows if the unitary representation of $H$ 
on $L^2_0(\Ga\backslash G)$ does not weakly contain 
the trivial representation $\mathbf{1}_H$ (\cite{cw,schmidt}). 
By \cite[Proposition 4.1]{bv} the latter holds true as long as $H$ is non-amenable. 

For the second assertion, we note that the assumption $\Ga \cap \cZ(G) \cap H =\{e\}$ 
is equivalent to that the action $\cZ(\Ga)\acts G/H$ is free. 
Here we note that
\[
\Ga \cap \cZ(G) = \cZ(\Ga) = \{ t\in\Ga : [\Ga : C_\Ga(t)]<\infty\},
\]
by the Borel density theorem (applied to the lattice $C_\Ga(t)$). 
Let $\cC_\infty$ denote the family of infinite index subgroups of $\Ga$ 
and $K:=\bigsqcup_{\La\in\cC_\infty}\Ga/\La$.
To prove that $\Ga\ltimes L^\infty(G/H)$ is a factor, take a unitary central element $w$, 
and expand it as $w=\sum_{t\in\Ga}\lambda(t)w^t$.
Since $\cZ(\Ga)\acts G/H$ is free, one has $w^t=0$ for all $t\in\cZ(\Ga)$ except 
that $w^e\in\IC1$ (by ergodicity). 
If $w\neq w^e$, then by the proof of Theorem~\ref{thm}, 
there is a normal $\Ga$-equivariant conditional expectation 
from $L^\infty(K\times G/H)$ onto $L^\infty(G/H)$. 
By Lemma~\ref{lem:ra}, this gives rise to a normal $H$-invariant state $\psi$ on 
$L^\infty(\bigsqcup_{\La\in\cC_\infty}\La\backslash G)$, 
which in turn provides a non-zero $H$-invariant vector 
in $L^2(\bigsqcup_{\La\in\cC_\infty}\La\backslash G)$. 
But such a vector is also $G$-invariant by Moore's ergodicity theorem, 
in contradiction with the fact that $\La$'s have infinite covolume in $G$. 

Next, we prove that the factor $\Ga\ltimes L^\infty(G/H)$ is full. 
The case of a rank one Lie group is already covered 
by Corollary~\ref{cor} and Example~\ref{exa:hyp} 
(note that the essential freeness assumption in Corollary~\ref{cor} is 
only used to assure factoriality of $\Ga\ltimes L^\infty(X)$ 
and $\La\ltimes L^\infty(X_0)$, and thus can be dispensed with by the above result).
Thus we may assume that $G$ has Kazhdan's property $\mathrm{(T)}$. 
Then, by \cite{cowling,moore}, the unitary representation of $H$ on 
$\bigoplus_{\La\in\cC_\infty}L^2(\La\backslash G)$ 
does not weakly contain $\mathbf{1}_H$.
(We note that we can avoid the use of this heavy machinery if $H$ 
has a non-compact subgroup with relative property $(\mathrm{T})$.)
This means that $L^\infty(\bigsqcup_{\La\in\cC_\infty} \La\backslash G)$ 
does not admit an $H$-invariant state (see \cite{schmidt}), and hence by Lemma~\ref{lem:ra} 
the action $\Ga\acts G/H$ is not amenable relative to $\cC_\infty$.
Now, suppose for a contradiction that the factor $\Ga\ltimes L^\infty(G/H)$ is not full and 
there is a unitary central sequence $(w_n)_n$ in $\Ga\ltimes L^\infty(G/H)$ 
such that $\sum_{t\in\cZ(\Ga)} |w_n^t|^2\to0$ (\cite[Lemma 5.1]{hi}). 
Then, by Theorem~\ref{thm}, this would imply $1=\sum_{t\in\Ga} |w_n^t|^2\to0$, 
which is absurd. 
%
\end{proof}

\begin{exa}
Consider the linear action of $\Ga:=\SL(n,\IZ)$ on $\IR^n$. 
Since the action extends to a measure preserving and 
essentially transitive action of $G:=\SL(n,\IR)$, it is isomorphic 
to the action $\Ga\acts G/H$, where $H\cong\SL(n-1,\IR)\ltimes\IR^{n-1}$ 
is the stabilizer subgroup of $G\acts\IR^n$ at $(1,0,\ldots,0)^T \in \IR^n$. 
If $n\geq 3$, then $H$ is non-amenable and $\SL(n,\IZ)\ltimes L^\infty(\IR^n)$ 
is a full factor by Theorem~\ref{thm2}. 
(If $n\le 2$, then $H$ is amenable and so is the action $\Ga\acts\IR^n$, 
see \cite[Remark 4.2]{pv}.)
Moreover, since $\SL(n,\IZ)$-action is measure preserving, 
$\SL(n,\IZ)\ltimes L^\infty(\IR^n)$ is a type $\mathrm{II}_\infty$ factor 
with a continuous trace-scaling $\IR_+^\times$-action, coming from 
the diagonal $\IR_+^\times$-action on $\IR^n$ by multiplication. 
The corresponding crossed product is a 
type $\mathrm{III}_1$ full factor that is isomorphic to $\SL(n,\IZ)\ltimes L^\infty(S^{n-1})$, 
where $\SL(n,\IZ)$ acts naturally on the sphere $S^{n-1}\cong(\IR^n\setminus\{0\})/\IR_+^\times$ 
(which is also isomorphic to a homogeneous space of $\SL(n,\IR)$). 
Indeed, since the $\IR_+^\times$-action is smooth and commutes with the $\SL(n,\IZ)$-action, 
one has natural isomorphisms 
\begin{align*}
(\SL(n,\IZ)\ltimes L^\infty(\IR^n)) \rtimes \IR_+^\times
 &\cong \SL(n,\IZ)\ltimes(L^\infty(S^{n-1} \times \IR_+^\times) \rtimes \IR_+^\times) \\
 &\cong \SL(n,\IZ)\ltimes(L^\infty(S^{n-1}) \otimes \IB(L^2(\IR_+^\times))) \\
 &\cong  (\SL(n,\IZ)\ltimes L^\infty(S^{n-1})) \otimes \IB(L^2(\IR_+^\times)).
\end{align*}
\end{exa} 

\begin{rem}\label{rem:ip} Adrian Ioana and Jesse Peterson kindly pointed 
out to the author that the action $\Ga\acts G/H$ as in Theorem~\ref{thm2} is essentially 
free, and moreover that for any connected Lie group $G$ and a closed subgroup $H\le G$ 
with the normal core $N:=\bigcap_{g\in G} gHg^{-1}$, 
the action $G/N\acts G/H$ is essentially free. 
Indeed, this follows from the fact that 
$\bigcap_{i=1}^{d+1} g_iHg_i^{-1}=N$ for a.e.\ $(g_1,\ldots,g_{d+1})\in G^{d+1}$, 
where $d=\dim G$. 
This fact ought to be well-known, but since we did not find a reference, 
we sketch a proof here (see also \cite[7.1]{cp}). Since $G$ is a connected Lie group, 
for any closed subgroup $K$ whose connected component is not normal in $G$, 
one has $\dim(K\cap gKg^{-1})<\dim K$ for a.e.\ $g\in G$. It follows that 
the connected component of $\bigcap_{i=1}^d g_iHg_i^{-1}$ is normal 
in $G$ almost surely (note that zero-dimensional subgroups are discrete and countable), 
and hence $\bigcap_{i=1}^{d+1} g_iHg_i^{-1}$ is normal in $G$ almost surely.
\end{rem}


\begin{thebibliography}{CFW}
%
\bibitem[A-D]{delaroche} C. Anantharaman-Delaroche;
Syst\`emes dynamiques non commutatifs et moyennabilit\'e. 
\emph{Math. Ann.} \textbf{279} (1987), 297--315. 
%
\bibitem[AH]{ah} H. Ando and U. Haagerup;
Ultraproducts of von Neumann algebras. 
\emph{J. Funct. Anal.} \textbf{266} (2014), 6842--6913. 
%
%
\bibitem[BV]{bv} M. E. B. Bekka and A. Valette; 
Lattices in semi-simple Lie groups, and multipliers of group $\mathrm{C}^*$-algebras. 
\emph{Recent advances in operator algebras (Orl\'eans, 1992). Ast\'erisque} No. \textbf{232} (1995), 67--79. 
%
%
%
\bibitem[Ch]{choda} M. Choda; 
Inner amenability and fullness. 
\emph{Proc. Amer. Math. Soc.} \textbf{86} (1982), 663--666. 
%
\bibitem[Co]{connes} A. Connes; 
Almost periodic states and factors of type $\mathrm{III}_1$. 
\emph{J. Funct. Anal.} \textbf{16} (1974), 415--445. 
%
\bibitem[CFW]{cfw} A. Connes, J. Feldman, and B. Weiss;
An amenable equivalence relation is generated by a single transformation. 
\emph{Ergodic Theory Dynam. Systems} \textbf{1} (1981), 431--450. 
%
\bibitem[CW]{cw} A. Connes and B. Weiss;
Property T and asymptotically invariant sequences. 
\emph{Israel J. Math.} \textbf{37} (1980), 209--210. 
%
\bibitem[Cow]{cowling} M. Cowling; 
Sur les coefficients des repr\'esentations unitaires des groupes de Lie simples. 
\emph{Lecture Notes in Mathematics} 739, pp. 132--178. Springer, Berlin, 1979.
%
\bibitem[CP]{cp} D. Creutz and J. Peterson; 
Stabilizers of ergodic actions of lattices and commensurators.
\emph{Trans. Amer. Math. Soc.} To appear. arXiv:1303.3949
%
\bibitem[HI]{hi} C. Houdayer and Y. Isono; 
Bi-exact groups, strongly ergodic actions and group measure space 
type $\mathrm{III}$ factors with no central sequence.
\emph{Comm. Math. Phys.} To appear. arXiv:1510.07987.
%
\bibitem[HR]{hr} C. Houdayer and S. Raum; 
Asymptotic structure of free Araki-Woods factors. 
\emph{Math. Ann.} \textbf{363} (2015), 237--267. 
%
%
\bibitem[Io]{ioana} A. Ioana; 
Strong ergodicity, property (T), and orbit equivalence rigidity for translation actions.
\emph{J. Reine Angew. Math.} To appear. 	arXiv:1406.6628
%
\bibitem[JS]{js} V. F. R. Jones and K. Schmidt; 
Asymptotically invariant sequences and approximate finiteness. 
\emph{Amer. J. Math.} \textbf{109} (1987), 91--114. 
%
\bibitem[Moo]{moore} C. C. Moore;
Exponential decay of correlation coefficients for geodesic flows. 
\emph{Group representations, ergodic theory, operator algebras, 
and mathematical physics (Berkeley, Calif., 1984)}, 163--181, 
Math. Sci. Res. Inst. Publ., 6, Springer, New York, 1987. 
%
%
\bibitem[Oz1]{kurosh} N. Ozawa; 
A Kurosh-type theorem for type $\mathrm{II}_1$ factors. 
\emph{Int. Math. Res. Not.} \textbf{2006}, Art. ID 97560, 21 pp.
%
\bibitem[Oz2]{relhyp} N. Ozawa; 
Boundary amenability of relatively hyperbolic groups.
\emph{Topology Appl.}, \textbf{153} (2006), 2624--2630.
%
\bibitem[PP]{pp} M. Pimsner and S. Popa;
Entropy and index for subfactors. 
\emph{Ann. Sci. \'Ecole Norm. Sup. (4)} \textbf{19} (1986), 57--106. 
%
\bibitem[PV]{pv} S. Popa and S. Vaes; 
Cocycle and orbit superrigidity for lattices in $\SL(n,{\mathbb R})$ acting on homogeneous spaces. 
\emph{Geometry, rigidity, and group actions}, 419--451, Univ. Chicago Press, 2011.
%
\bibitem[Sa]{sako} H. Sako; 
The class $S$ as an ME invariant. 
\emph{Int. Math. Res. Not. IMRN} \textbf{2009}, 2749--2759.
%
\bibitem[Sch]{schmidt} K. Schmidt; 
Amenability, Kazhdan's property T, strong ergodicity and invariant means for ergodic group-actions. 
\emph{Ergodic Theory Dynam. Systems} \textbf{1} (1981), 223--236. 
%
\bibitem[Zi]{zimmer:invent} R. J. Zimmer;
Hyperfinite factors and amenable ergodic actions. 
\emph{Invent. Math.} \textbf{41} (1977), 23--31. 
%
%
\end{thebibliography}
\end{document}